\documentclass[11pt,reqno]{amsart}
\usepackage{amsmath}
\usepackage{amsfonts}
\usepackage{amsthm}
\usepackage{amssymb}
\usepackage{mathrsfs}
\usepackage{array}
\usepackage{latexsym}
\usepackage{verbatim}
\usepackage{hyperref}
\usepackage{a4wide}
\usepackage{enumerate}
\usepackage{lipsum}
\usepackage[all]{xy}
\usepackage{tikz-cd}  

\numberwithin{equation}{section}

\renewcommand{\thetheoremName}

\newcommand{\IC}{\mathbb{C}}
\newcommand{\IF}{\mathbb{F}}

\newcommand{\IN}{\mathbb{N}}
\newcommand{\IG}{\mathbb{G}}

\newcommand{\IQ}{\mathbb{Q}}
\newcommand{\IR}{\mathbb{R}}
\newcommand{\IZ}{\mathbb{Z}}

\newcommand{\calF}{\mathcal{F}}
\newcommand{\calG}{\mathcal{G}}
\newcommand{\calH}{\mathcal{H}}

\newcommand{\calO}{\mathcal{O}}

\newcommand{\calT}{\mathcal{T}}

\newcommand{\calV}{\mathcal{V}}

\newcommand{\calZ}{\mathcal{Z}}

\newcommand{\im}{\mathfrak{m}}

\def\PGL{\mathrm{PGL}}

\def\Hom{\mathrm{Hom}}
\def\End{\mathrm{End}}

\def\GL{\mathrm{GL}}
\def\rank{\mathrm{rank}}
\def\Gal{\mathrm{Gal}}

\def\SL{\mathrm{SL}}

\def\Spec{\mathrm{Spec}}
\def\Spf{\mathrm{Spf}}

\def\codim{\mathrm{codim}}

\def\Res{\mathrm{Res}}

\DeclareMathOperator\Stab{Stab}
\DeclareMathOperator\Supp{Supp}

\newtheorem{theorem}{Theorem}[section]
\newtheorem*{theorem*}{Theorem}
\newtheorem*{corollary*}{Corollary}

\newtheorem{question}[theorem]{Question}
\newtheorem{defn}[theorem]{Definition}
\newtheorem{lemma}[theorem]{Lemma}
\newtheorem{prop}[theorem]{Proposition}
\newtheorem{corollary}[theorem]{Corollary}
\newtheorem{conj}{Conjecture}

\theoremstyle{definition}

\theoremstyle{remark}

\newtheorem{remark}[theorem]{Remark}

\thanks{The author's research was partially funded by START-prize Y-966 of the Austrian Science Fund (FWF) under P.I. Harald Grobner}
\begin{document}
\title{Unlikely intersections on the $p$-adic formal ball}
\author{Vlad Serban}

\bibliographystyle{alpha}

\begin{abstract}
We investigate generalizations along the lines of the Mordell--Lang conjecture of the author's $p$-adic formal Manin--Mumford results for $n$-dimensional $p$-divisible formal groups $\calF$. In particular, given a finitely generated subgroup $\Gamma$ of $\calF(\overline{\IQ}_p)$ and a closed subscheme $X\hookrightarrow \calF$, we show under suitable assumptions that for any points $P\in X(\IC_p)$ satisfying $nP\in\Gamma$ for some $n\in\IN$, the minimal such orders $n$ are uniformly bounded whenever $X$ does not contain a formal subgroup translate of positive dimension. In contrast, we then provide counter-examples to a full $p$-adic formal Mordell--Lang result. Finally, we outline some consequences for the study of the Zariski-density of sets of automorphic objects in $p$-adic deformations. Specifically, we do so in the context of the nearly ordinary $p$-adic families of cuspidal cohomological automorphic forms for the general linear group constructed by Hida. 
\end{abstract}

\maketitle
\section{Introduction}
The classical Mordell--Lang conjecture formulates an unlikely intersection principle for subvarieties $V\subset G$ of a semi-abelian variety in characteristic $0$. Given a finitely generated subgroup $\Gamma< G(\mathbb{C})$, it states that the divisible hull $\Gamma^{div}:=\{g\in G(\mathbb{C})\vert g^n\in \Gamma \text{ for some } n\}$ of $\Gamma$ cannot have Zariski-dense intersection with a reduced and irreducible subvariety $V$ unless the latter is the translate of an algebraic subgroup of $G$ by an element in $\Gamma^{div}$. The conjecture's proof was completed by McQuillan \cite{McQuillan1995DivisionPO}, building crucially on work of Faltings, Hindry, Raynaud, Vojta and others (see, e.g.,\cite{Faltings1983, Raynaud1983,Hindry1988,Buium1992IntersectionsIJ,Vojta1996IntegralPO}). 
For tori $G=\mathbb{G}_m^n$, the conjecture had been settled previously by Laurent \cite{MR767195}, relying on work of Liardet \cite{AST_1975Liardet} and Lang \cite{MR0130219, MR0190146}.\par
In this paper, we investigate ceretain $p$-adic analytic versions of the conjecture, leading to unlikely intersection results in the context of $p$-adic formal groups. More precisely, consider the formal multiplicative group $\widehat{\IG}_m^n:=\Spf(R[[X_1,\ldots,X_n]])$ over a complete, discretely valued subring $R$ of the $p$-adic complex numbers or more generally consider an $n$-dimensional $p$-divisible formal Lie group $\mathcal{F}$ over $R$.  We let $X\hookrightarrow\Spf(R[[X_1,\ldots,X_n]])$ denote an irreducible closed formal subscheme and $\Gamma<\calF(\IC_p)$ a finitely generated subgroup, and let $\Gamma^{div}$ denote the divisible hull of $\Gamma$ as above. We examine when $X\cap \Gamma^{div}$ can be Zariski-dense in $X$ and whether $X$ has to be the translate of a formal subtorus by a point in $\Gamma^{div}$. 
When $\Gamma=\{0\}$, this reduces to examining a $p$-adic version of the so-called Manin--Mumford conjecture, when $\Gamma^{div}$ consists of the torsion points of the formal group. Our first result, Theorem \ref{thm:generalbackward}, is then a generalization of the $p$-adic formal Manin--Mumford results established by the author \cite{Serban:2016aa, serban2020padic}: 
\begin{theorem}\label{thm:mordellintro}
Let $X\hookrightarrow\calF$ denote a closed affine formal subscheme of an $n$-dimensional $p$-divisible commutative formal group and let $\Gamma\subset\calF(\overline{R[1/p]})$ be a finitely generated $\IZ_p$-module, where $\IZ_p$ acts as an endomorphism via the inclusion $\IZ_p\subset \End(\calF)$. Let $\Sigma\subseteq \Gamma^{div}$ be any subset of the divisible hull of $\Gamma$. Then exactly one of the following occurs: 
\begin{enumerate}
    \item There exists a fixed positive integer $m$ such that any $\alpha\in \Sigma\cap X(\mathbb{C}_p)$ satisfies $[p^m](\alpha)\in \Gamma$. Moreover, if $\Sigma$ is discrete in the $p$-adic topology on the unit ball $\im_{\IC_p}^n$, then for $\varepsilon >0$ small enough the same conclusion applies for any $\alpha\in \Sigma$ with $d(\alpha, X)\leq \varepsilon$.
    \item The formal subscheme $X$ contains the translate of a formal subgroup of $\calF$ of positive dimension by an element of $\Sigma$ (over a finite extension of $R$).
\end{enumerate}
\end{theorem}

Moreover, we obtain a stronger result in the case of the formal multiplicative group. Namely we find in Theorem \ref{thm:backwardtorus} that for a closed formal subscheme $X\hookrightarrow\widehat{\IG}_m^n$ and under the same assumptions on $\Gamma$, in addition to Theorem \ref{thm:mordellintro} if $\Gamma^{div}$ is Zariski-dense in $X$ then either: 
\begin{enumerate}
    \item $X$ \textbf{is} a formal subtorus translate, or
    \item the order $\{\min(n): g\in X\cap \Gamma^{div}, [n](g)\in \Gamma \} $ is uniformly bounded.
\end{enumerate}
We note that Theorems \ref{thm:generalbackward} and \ref{thm:backwardtorus} exhibit more generally for backward orbits under the multiplication-by-$p$ endomorphism a phenomenon known for torsion points as the Tate--Voloch conjecture: backward iterates tend to either lie squarely on a formal subscheme $X$ or are uniformly bounded away in the ultrametric distance. \par
One may interpret the above as rigidity results for $p$-adic formal power series.  Such techniques appear in work of Monsky \cite{MR614398}, together with Iwasawa-theoretic applications (see, e.g., \cite{MonskyCuoco}). Moreover, Hida proved and employed rigidity results to characterize CM families of (Hilbert) modular forms by the complexity of their Hecke fields (see, e.g., \cite{HidamodularHeckefields,Hidatranscendence,HidaHeckeFieldsGrowth}). Further rigidity results were established by Berger \cite{BergerRigidity}, as well as the author \cite{Serban:2016aa,serban2020padic}. Inspired by Hida's approach, Stubley \cite{Stubley} also uses such methods to establish finiteness results for non-CM Hilbert modular forms of partial weight one in $p$-adic families. In fact, the rigidity results listed above can be related to some version of a Manin--Mumford type theorem concerning torsion points on formal tori or more general formal groups. Part of the motivation behind this paper was to uncover interesting new questions and applications by going beyond torsion points.\par
However, we discovered an obstacle to such ambitions which highlights the subtlety of rigidity results for $p$-adic analytic functions: namely, we find counter-examples in Proposition \ref{prop:counterexample} to $p$-adic formal formulations of Mordell--Lang: 
\begin{prop}
 Let $\calF^2$ denote two copies of a one-dimensional $p$-divisible formal group over $\IZ_p$ of finite height. Let $\alpha_1,\alpha_2\in p\IZ_p$ and suppose for simplicity their $p$-adic valuations $v_p(\alpha_1)=v_p(\alpha_2)$ agree. Let $[n]$ denote the multiplication-by-$n$ endomorphism in the formal group. Then there exists $\phi\in \IZ_p[[X,Y]]$ such that for an infinite set $S$ of pairs $(n,m)\in\mathbb{N}^2$ one has:
$$\phi([n](\alpha_1),[m](\alpha_2))=0$$
and the set $S$ cannot be covered by any finite union of lines, implying that the zeroes of $\phi$ do not come from a formal subgroup translate. 
 \end{prop}
These counter-examples are inspired by investigations of the dynamical Mordell--Lang Conjecture (see e.g., \cite{BenedettoGhiocaGapPrinciple, GhiocaBellTuckerdynamicalMLbook}). The results above therefore seem to indicate a certain dichotomy in the nonarchimedean setting. Namely, for backward iterates under non-invertible endomorphisms of formal groups, Galois symmetry, ramification considerations and the fact that the underlying rings of functions are Noetherian suffices to push through unlikely intersection results beyond the algebraic setting. However, for forward iterates, where for instance one does not have such Galois symmetry at one's disposal, the results do not typically extend beyond algebraic functions. A natural question, then, is whether the results for backward iterates along the lines of \ref{thm:mordellintro} hold for more general maps not necessarily known to be endomorphisms of formal groups. For instance, one might consider preperiodic points for the nonarchimedean dynamical systems studied by Lubin \cite{Lubindynamical}, even though such dynamical systems are in many instances known or believed to originate from formal groups (see, e.g., \cite{Sarkis2010,JoelSpecterdynamical,Hsia_Liramification,BergerLubinAMS}). We plan to address this in future work.\par

Finally, in Section \ref{section:applications} we discuss applications to $p$-adic variational techniques, where spaces $\Spf(R[[X_1,\ldots,X_n]])$ occur naturally as components of $p$-adic weight spaces $\Spf(\Lambda)$ parametrizing $p$-adic families of automorphic forms or as parametrizing deformations of Galois representations. What is more, there is a class of \emph{special points} on these spaces: on weight space we call them \emph{locally algebraic} weights. It is at a subset of these that $p$-adic families may interpolate classical automorphic forms. Similarly, there is a notion of special or \emph{automorphic} points on deformation spaces. The locally algebraic points can be viewed as sitting inside the divisible hull of a finitely generated subgroup of $\widehat{\IG}_m^n(\overline{\IQ}_p)$. 
One may thus hope to leverage unlikely intersection results to show that, in instances when this is expected, Hecke eigensystems related to automorphic forms do not deform into $p$-adic families with a dense set of \emph{classical} Hecke eigensystems coming from automorphic forms. This would inform the expectation that cuspidal Hecke eigensystems that do not arise from functoriality constructions are sparse for reductive groups $G$ with discrete series defect $\ell_0>0$. See for instance work of Ash, Pollack and Stevens \cite{AshPollackStevens} where this is established in some examples for $\GL_3(\IZ)$ or see \cite{Serban2018} for examples of families interpolating finitely many Bianchi modular forms. Work of Calegari--Mazur and Childers \cite{MR2461903, Childers2021GaloisDS} establishes similar results for certain deformations of Galois representations. 
\par 
In order to better illustrate an approach based on unlikely intersections, let $\calH$ denote a component of a nearly ordinary $p$-adic family of cuspidal cohomological automorphic forms constructed by Hida \cite{MR1313784,HidaGLN} for $\GL_2$ or more generally $\GL_N$ over $F$ an arbitrary number field. Assuming $F$ has $r_2>0$ complex places, we consider a component $\calV$ of the support of $\calH$ in weight space,  which on each of the finitely many components may be viewed as a formal subscheme $\calV\hookrightarrow\widehat{\IG}_m^n=\Spf(R[[X_1,\ldots,X_n]])$, conjecturally of codimension $r_2$. The weights at which $\calH$ may interpolate classical automorphic forms lie in a subset $\Delta$ (the arithmetic, locally parallel subset) of the locally algebraic weights. An optimistic strategy to prove finiteness or more generally lack of density of the classical automorphic forms interpolated by the family, similar to the way the Mordell Conjecture is deduced from Mordell--Lang for Jacobians of curves, would therefore proceed as follows: 
\begin{enumerate}
    \item \textit{(unlikely intersection result)} If $\calV \cap \Delta$ is infinite (resp. Zariski-dense in $\calV$), then $\calV$ contains (resp. is) the translate of a positive-dimensional formal subtorus of $\widehat{\IG}_m^n$. 
    \item \textit{(exclusion step)} Show that the support $\calV$ cannot be (or contain) a positive-dimensional formal subtorus (normalizing so that $\calV\hookrightarrow\widehat{\IG}_m^n$ passes through the origin).
\end{enumerate}
However, given the Mordell--Lang counter-examples, the conclusion in the first step is too strong to hold in general. Rather we have (see Corollary \ref{cor:weights}): 

\begin{corollary}\label{cor:weightsintro}
Let $\calV$ denote an irreducible component of the Zariski-closure in weight space $\Spf(\Lambda)$ of a set of locally algebraic points. Then either $\calV$ is the translate of a formal subtorus by a locally algebraic point, or there exists a fixed $k>0$ such that integer weights are Zariski-dense in $[p^k](\calV)$, where $[p^k](\mathbf{\alpha})=(1+\mathbf{\alpha})^{p^k}-1$ denotes the corresponding endomorphism of the formal torus.
\end{corollary}

Furthermore, it is in general unclear how to carry out the second step of excluding a component that is a formal torus: at the present time we do not know enough about the embedding $\calV\hookrightarrow\widehat{\IG}_m^n$ to successfully apply this strategy in general, even for components where finiteness (or lack of density) of weights in $\calV \cap \Delta$ presumably holds. For instance, just establishing the codimension of $\calV$ to be $r_2$ as conjectured by Hida would amount to showing a $\GL_N$-version of the Leopoldt conjecture (see \cite{HidaGLN,Khare2014PotentialAA}). In addition to that, it actually \emph{can} happen in our context that the component $\calV$ is a formal subtorus and thus a dense set of classical points is interpolated. \par
For example, consider the case of $\GL_2$ over an imaginary quadratic field $F$. For a split prime $p$, a Bianchi cuspidal modular form which contributes to the $p$-ordinary part of the cohomology of a congruence subgroup of $\PGL_2(\calO_F)$ deforms to a $p$-adic family whose support here is known (\cite[Theorem 6.2]{MR1313784}) to have codimension $r_2=1$ in weight space. The latter is a finite union of two-dimensional formal disks $\Spf(R[[X_1,X_2]])$. In \cite{Serban2018} we obtain finiteness results for classical automorphic points on such ordinary families via the two-step strategy above. Here, the Zariski-closure of algebraic weights in $\Delta$ is simply the diagonal, so that the unlikely intersection result for infinite $\calV \cap \Delta$ forces the support $\calV$ to be either: 
\begin{enumerate}[(a)]
    \item the diagonal $D$ in $\Spf(R[[X_1,X_2]])$ or 
    \item a one-dimensional formal torus of finite intersection with $D$. 
\end{enumerate}
For concrete $p$-adic families, both of these options can be excluded (provided finiteness indeed holds) by checking computationally for Hecke eigensystems at finite levels which ought to arise via specialization at weights in $\calV \cap \Delta$. However, base-changing a $p$-adic family over $\IQ$ to $F$ we obtain a family whose support indeed contains the diagonal $D$ and interpolates a dense set of automorphic forms. \par

Similarly, Ash, Pollack and Stevens \cite{AshPollackStevens} find in their terminology some \emph{arithmetically rigid} eigenpackets for $\GL_3/\IQ$ by an algebraic rigidity argument \cite[Theorem 8.4]{AshPollackStevens} together with explicit computation (see \cite[Section 9]{AshPollackStevens}). Here, it is families arising via symmetric-square lifts of modular forms that violate arithmetic rigidity. \par
It would be interesting to investigate further to what extent examples where density does hold must arise via functoriality constructions  (see, e.g., \cite[Conjecture 0.1]{AshPollackStevens} and \cite[Conjecture 1.3.]{MR2461903}). To that end, we establish in Section \ref{section:applications} some obstructions to density for Hida families for $\GL_2$ over arbitrary number fields (see Theorem \ref{thm:hidafamapplication} and Proposition \ref{prop:generalfields}) and discuss a few possible further avenues for applications. \par
Even though progress on such questions will require additional work given some of the aforementioned limitations of rigidity methods, we hope that they provide a useful tool to tackle some of the natural questions arising from $p$-adic variational techniques. 

\section{Unlikely intersections beyond formal Manin--Mumford}\label{section:rigidity}
\subsection{Goals and Notations} 
In this section, we prove our main unlikely intersection results, which may also be interpreted as rigidity results for formal $p$-adic power series. We denote by $\IC_p,\calO_{\IC_p}$ and $\im_{\IC_p}$ the $p$-adic complex numbers together with their valuation ring and maximal ideal, respectively. We let throughout $R\subset\calO_{\IC_p}$ denote a complete discrete valuation ring of mixed characteristic $(0,p)$ and write $K$ for its field of fractions. Our main goal is to study which closed, affine formal subschemes over $R$ 
$$X\hookrightarrow\Spf(R[[X_1,\ldots,X_n]])$$ 
contain a Zariski-dense set of \emph{special points}. The special points we consider are subsets of the divisible hull of a finitely generated subgroup of some $n$-dimensional commutative $p$-divisible formal group law over $R$. Our main concern in this paper is when the ambient formal group is an $n$-dimensional formal torus $\widehat{\IG}_m^n=\Spf(R[[X_1,\ldots,X_n]])$ and we wish to investigate a $p$-adic formal version of the multiplicative Mordell-Lang Conjecture. Namely, we want to examine under which assumptions formal subschemes $X$ with a Zariski-dense set of special points have to themselves be \emph{special subschemes}: schemes whose components are translates of a formal subtorus by a special point. \par
However, it makes sense to consider more general formal groups: for instance beyond $\widehat{\IG}_m^n$ we may consider powers $\calG^n$ of an arbitrary finite height one-dimensional formal group $\calG$ over $R$ such as a Lubin--Tate formal group and study the same question. Or one may even consider general $n$-dimensional $p$-divisible formal groups $\calF$ such as the completion along the identity section of an abelian scheme over $R$. Thus, a number of our results will be established in some greater generality beyond the formal multiplicative group case. \par
 We shall slightly abuse notations and write $\calF(A)$ for the points of $\calF$ over the topologically nilpotent elements of an $R$-algebra $A$. 
In this $p$-adic setting, given generators $\{\gamma_1,\ldots \gamma_r\}\in \calF(\IC_p)$ we may write the special points as follows: 
using the endomorphisms $\IZ_p\hookrightarrow\End(\calF)$ we set (with addition denoting the formal group law):
$$\Gamma:=\{\sum_{j=1}^n[a_{j}](\gamma_j)\vert a_j\in \mathbb{Z}_p\},$$
denoting by $[a](\mathbf{X})$ for $a\in \IZ_p$ the corresponding endomorphism of $\calF$. That is, we may even consider the $\IZ_p$-module generated by $\{\gamma_1,\ldots \gamma_r\}$ rather than just the $\mathbb{Z}$-module generated, and we do so unless otherwise specified.
Throughout we also employ boldface letters to abbreviate for multiple variables or multiple components of elements. We then may write the special points as
$$\Gamma^{div}:=\bigcup_{i=0}^\infty \Gamma^{-i},$$
where $\Gamma^{-i}=\{\mathbf{x}\in \im_{\IC_p}^n\vert [p^i](\mathbf{x})\in \Gamma\}$ and $[p^i](\mathbf{X})\in R[[\mathbf{X}]]$ is the multiplication-by-$p$ endomorphism in the formal group. We also typically assume the generators $\gamma_i$ are algebraic over $K$. \par
Our study of the Zariski-closure of sets of special points may then be split into two parts: we first study the backwards iterates under multiplication-by-$p$ of elements in $\Gamma$. These are defined over increasingly ramified field extensions, and we prove that in the absence of a formal subgroup translate in $X$ we have that $X(\mathbb{\IC}_p)\cap\Gamma^{div}\subset\Gamma^{-m}$ for some $m\in \IN$. When $\Gamma=\{0\}$, this is equivalent to a formal $p$-adic version of the Manin--Mumford Conjecture for curves, though for formal tori we prove a  generalization of the full Manin--Mumford statement. Second, we examine the case when the special points just consist of a finitely generated $\IZ$ or $\IZ_p$-module $\Gamma$. We shall see that in this second case in contrast unlikely intersection results of Mordell--Lang type do not typically hold.  
\subsection{Backward orbits}\label{subsec:backward}
We consider backward iterates of $\Gamma$ under the multiplication-by-$p$ map. Similar to the Manin--Mumford type results occurring for $\Gamma=\{0\}$ and established by the author in this formal setting in \cite{serban2020padic}, we observe for backward iterates that a version of the Tate--Voloch Conjecture again holds here: points in a backward orbit either lie squarely on a closed formal subscheme or are uniformly bounded away from it for a suitable $p$-adic distance on the unit polydisk. To that end, for a closed affine formal subscheme $X\hookrightarrow \calF$ given by $X=\Spf(R[[X_1,\ldots, X_n]]/I)$ for some ideal of definition $I$, we define for any $P\in\calF(\IC_p)$ the distance function
$$d(P,X):=\max_{f\in I}(\vert f(P)\vert_p),$$
using the $p$-adic absolute value $\vert-\vert_p$ (which we may normalize so that $\vert p\vert_p=1/p$). The distance depends on the choice of normalization of the $p$-adic absolute value and the choice of coordinates on the formal Lie group. However, the distance from a torsion point to a formal subscheme is invariant under automorphisms of the ambient formal Lie group, and this suffices for our statements to be well-defined. We also record the following: 
\begin{lemma}\label{lemma:capepsilon}
Let $X=\Spf(R[[X_1,\ldots, X_n]]/I)$ for some ideal $I$. Then $I$ is finitely generated and, writing $I=(f_1,\ldots,f_k)$, the distance to any point $P\in\calF(\IC_p)$ can be computed on generators as $d(P,X)=\max_{i=1}^k(\vert f_i(P)\vert_p)$. 
\end{lemma} 
\begin{proof}The ring $R[[X_1,\ldots, X_n]]$ is Noetherian and thus $I$ is finitely generated. The fact that it suffices to compute the distance on the generators $f_i$ is a feature of the ultrametric setting: writing any $f\in I$ as a sum $\sum_{j=1}^s g_j f_{i_j}$ for $i_j\in\{1,\ldots,k\}$ we have that $\vert \sum_{j=1}^s g_j f_{i_j}(P)\vert_p\leq \max_{j=1}^s\vert g_j f_{i_j}(P) \vert_p\leq \max_{j=1}^s\vert f_{i_j}(P) \vert_p$. 
\end{proof}

We shall employ a similar strategy to \cite{serban2020padic} for the Manin--Mumford case: given a formal subscheme $X\hookrightarrow\calF$, we consider the points in $\calF(\IC_p)$ stabilizing $X$ under translation via the formal group law.
We therefore define the \emph{stabilizer} on the level of points via: 
$$\Stab(X)(A):=\{P\in\calF(A)\vert X(A)+_\calF P= X(A)\}$$
for a commutative nilpotent $R$-algebra $A$. When $\Stab(X)(\overline{K})$ is infinite, we show that the stabilizer yields a positive-dimensional closed formal subgroup of $\calF$ over a finite extension of $R$ whose translate is contained in $X$. If however $\Stab(X)(\overline{K})$ is finite, we shall deduce the result using the action of the Galois group $\Gal(K)$ on the backward orbits. Let $\calF[p^r]$ denote the $p^r$-torsion points on the formal group $\calF$. We have the following result, inspired by \cite[Lemma 2]{Volochintegrality}, concerning the Galois action: 

\begin{lemma} \label{lemma:Volochgeneral}
Let $\calF$ denote an $n$-dimensional commutative $p$-divisible formal group over $R$ a complete discretely valued subring of $\calO_{\IC_p}$. Let $K$ denote the field of fractions of $R$ and let $r>1$ be a fixed integer. Let $\Gamma\subset\calF(\overline{K})$ denote a finitely generated $\IZ_p$-module, where $\IZ_p$ acts as an endomorphism via the inclusion $\IZ_p\subset \End(\calF)$ and set $\Gamma^{-i}:=\{P\in\calF(\mathbb{C}_p)\vert [p^i]P\in \Gamma\}$. Finally, set $L:=K(\Gamma^{-r})$ and let $P\in \bigcup_{i=0}^\infty \Gamma^{-i}$ denote a point.
Assume that the least integer $n$ such that $[p^n](P)$ is defined over $L$ satisfies $n>r-1$. Then there exists $s\in \Gal(L)$ such that $s(P)-P\in \calF[p^r]\setminus\calF[p^{r-1}]$, where subtraction is performed using the formal group law of $\calF$. 

\end{lemma}
\begin{proof}
Define $P_i:=[p^{n-(r-1)-i}]P$ for $i\in \{1,\ldots,n-(r-1)\}$ and pick $s_1\in \Gal(L)$ such that $s_1([p^{n-1}]P)\neq[p^{n-1}]P$. Moreover set $s_i:=s_1^{p^{i-1}}$ and $Q_i:=s_i(P_i)-(P_i)$, noting that addition and subtraction are using the formal group law throughout the proof. We claim that 
$$Q_i\in \calF[p^r]\setminus\calF[p^{r-1}].$$
It then follows that $s_{n-r+1}\in \Gal(L)$ is the desired element since $P_{n-(r-1)}=P$, thereby establishing the lemma. We prove the claim by induction on $i$, noting that it holds for $i=1$, since $[p^r]Q_1=s_1([p^n]P)-[p^n]P=0$ and 
$[p^{r-1}]Q_1=s_1([p^{n-1}]P)-[p^{n-1}]P\neq 0$  by construction.
Assume now that the claim holds for $Q_i$. First observe that $s_i^j(P_i)-P_i=[j]Q_i$ for $j\geq1$: this is true for $j=1$ and follows easily by induction on $j$ since then
$$s_i^{j+1}(P_i)-P_i=s_i([j](Q_i))+s_i(P_i)-P_i=[j+1]Q_i, $$
using that $Q_i\in \calF[p^r]$ by induction and is therefore fixed by $s_i$.
Taking $j=p$ yields $s_i^p(P_i)-P_i=[p]Q_{i}\in \calF[p^{r-1}]$ by the induction hypothesis. We deduce that  $$[p]Q_{i+1}=[p](s_{i+1}(P_{i+1})-P_{i+1})=s_i^p([p]P_{i+1})-[p]P_{i+1}=s_i^p(P_i)-P_i=[p]Q_i.$$ 
This shows by induction that $Q_{i+1}\in \calF[p^r]$ and, provided $r>1$, that $Q_{i+1}\not\in \calF[p^{r-1}]$, concluding the proof.
\end{proof}
In particular, we shall utilize the following consequence of Lemma \ref{lemma:Volochgeneral}: 
\begin{lemma}\label{lemma:inertia} Let $\calF$ be a finite dimensional commutative $p$-divisible formal Lie group defined over $R$ and let $r>1$ as well as $\Gamma\subset\calF(\overline{K})$ a finitely generated $\IZ_p$-module, where $\IZ_p$ acts as an endomorphism via the inclusion $\IZ_p\subset \End(\calF)$, be fixed. Write $\Gamma^{-i}:=\{P\in\calF(\mathbb{C}_p)\vert [p^i]P\in \Gamma\}$ and set $L:=K(\Gamma^{-r})$. Then for a large enough integer $t$, there exists for any point $P\in \bigcup_{i=0}^\infty \Gamma^{-i}\setminus\Gamma^{-t}$ an automorphism $s\in \Gal(L)$ such that $s(P)-P\in \calF[p^r]\setminus \calF[p^{r-1}]$. 
\end{lemma}
\begin{proof}
It suffices to show we can choose $t\geq 0$ large enough so that any $P\in \bigcup_{i=0}^\infty \Gamma^{-i}\setminus \Gamma^{-t}$ has the property that $[p^{r-1}]P$ is not defined over $L$. The result then follows from applying Lemma \ref{lemma:Volochgeneral}. The existence of a large enough $t$ follows from ramification considerations in the fields $K(\bigcup_{i=0}^{r+j} [p]^{-i}(\gamma))$ for any $\gamma\in\Gamma$. Indeed, since the group law is defined over $K$, the field $L$ is a finite extension of $K$ and has a discrete valuation ring. Given $\alpha\in \Gamma^{-(r+j)}\setminus \Gamma^{-(r+j-1)}$, any point $\tilde{\alpha}\in \im_{\IC_p}^n$ with $[p](\tilde{\alpha})=\alpha$ has coordinates of valuation $v_p(\tilde{\alpha}_i)\leq v_p(\alpha_i)/2$. This can be established by inspection of the slopes of the Newton polygon of $[p](X)-\alpha$ after specializing to a coordinate, using that $[p](X)-\alpha=-\alpha+p\cdot X$ up to higher order terms. For $j$ large enough one of the coordinates $\alpha_i$ of $\alpha\in \Gamma^{-(r+j)}\setminus \Gamma^{-(r+j-1)}$ is therefore not in $K(\Gamma^{-(r+j-1)})$. The result follows.
\end{proof}
For any $Q\in \calF(\IC_p)$ we define the translate $T_QX$ on the level of points by $T_QX(\IC_p)=\{x-Q\vert x\in X(\IC_p)\}$. It is easy to show (see \cite[Section 2]{serban2020padic} for a proof of Lemma \ref{lemma:translate}) the following properties of translates: 
\begin{lemma}\label{lemma:translate}
The translate $T_QX$ is a closed affine formal subscheme $i_Q:T_QX\hookrightarrow \calF$ over the ring $R[Q]$, which is Noetherian for $Q\in\overline{K}$. 
\end{lemma}
\begin{lemma}\label{translateclose}
Let $\calF$ be a formal group over $R$ and let $X$ be a closed affine formal subscheme over $R$. Then for any $\sigma\in \Gal(K)$ we have
$$d(\alpha, T_{\sigma(\alpha)-\alpha}X)=d(\alpha, X).$$
\end{lemma}
\begin{proof}
Writing $\alpha=\sigma(\alpha)+(\alpha-\sigma(\alpha))$, we see that it suffices to show that $d(\sigma(\alpha), X)=d(\alpha, X)$. For any $\phi$ in the ideal of definition of $X$, we know that $\phi(\sigma(\alpha))=\sigma(\phi(\alpha))$. But $\vert\sigma(\phi(\alpha))\vert_p=\vert\phi(\alpha)\vert_p$ since $\sigma$ acts by continuous automorphisms, and thus the result follows. 
\end{proof}
Our last preparatory result will allow us to deal with the infinite stabilizer case: 
\begin{lemma}\label{lemma:infstab}
Let $X\hookrightarrow \calF$ denote a closed affine formal subscheme of an $n$-dimensional formal Lie group over $R$. If $\Stab(X)(\overline{K})$ is infinite, then the stabilizer is a formal subgroup scheme of $\calF$ over some finite extension $L$ of $K$ and has a positive-dimensional component. Moreover, its translate by any $\gamma\in X(\mathbb{C}_p)$ is contained in $X$. 
\end{lemma}
\begin{proof}
For the first part this is established as in \cite[Lemmas 2.2.,2.3]{serban2020padic}, with the only difference being that here we may not assume that $X$ passes through the origin. But since we have in this case that $\Stab(X)(\mathbb{C}_p)=\cap_{Q\in X}T_QX(\mathbb{C}_p)$ as sets, we obtain that $X$ contains a translate by $\gamma\in X(\mathbb{C}_p)$ of $\Stab(X)$.
\end{proof}

We are now ready to prove the first main result of this section: 
\begin{theorem}\label{thm:generalbackward}
Let $\calF$ denote an $n$-dimensional commutative $p$-divisible formal Lie group over $R$. Let $X\hookrightarrow\calF$ denote a closed affine formal subscheme and let $\Gamma\subset\calF(\overline{K})$ be a finitely generated $\IZ_p$-module, where $\IZ_p$ acts as an endomorphism via the inclusion $\IZ_p\subset \End(\calF)$. Let $\Sigma\subseteq \Gamma^{div}$ be any subset of the divisible hull of $\Gamma$. Then exactly one of the following occurs: 
\begin{enumerate}
    \item There exists a fixed positive integer $m$ such that any $\alpha\in \Sigma\cap X(\mathbb{C}_p)$ satisfies $[p^m](\alpha)\in \Gamma$. Moreover, if $\Sigma$ is discrete in the $p$-adic topology on the unit ball $\im_{\IC_p}^n$, then for $\varepsilon >0$ small enough the same conclusion applies for any $\alpha\in \Sigma$ with $d(\alpha, X)\leq \varepsilon$.
    \item The formal subscheme $X$ contains the translate of a formal subgroup of $\calF$ of positive dimension by an element of $\Sigma$ (over a finite extension of $R$).
\end{enumerate}

\end{theorem}
We remark that in particular taking $\Gamma$ to be the trivial group yields that there are only finitely many torsion points approaching a closed formal subscheme unless it contains a translate of a formal subgroup by a torsion point. Furthermore, taking $\Sigma$ to be the backward orbit of finitely many $\gamma\in\im_{\IC_p}^n$ under multiplication-by-$p$, we get a Manin--Mumford and Tate--Voloch type result for the more general set $\Sigma$ of special points, as now $\Sigma$ is discrete in the $p$-adic topology (however the full $\Gamma^{div}$ is in general not discrete). 

\begin{proof}
We may as well assume that $X$ is irreducible and proceed by induction on $\dim(X)$. If $\Stab(X)(\overline{K})$ is infinite, we conclude the result holds by Lemma \ref{lemma:infstab}. Under the assumption that $\Stab(X)(\overline{K})$ is finite, we may pick $r>1$ large enough so as to ensure that $\Stab(X)(\overline{K})\cap \calF[p^\infty]\subset \calF[p^{r-1}]$. By Lemma \ref{lemma:inertia} we may also fix $t\geq r$ large enough so that for any point $P\in \bigcup_{i=t}^\infty \Gamma^{-i}\setminus \Gamma^{-t}$ there exists an automorphism $s\in \Gal(L)$ such that $s(P)-P\in \calF[p^r]\setminus \calF[p^{r-1}]$. For any $\varepsilon>0$, we abbreviate $X(\varepsilon)=\{\zeta\in \Sigma\setminus\Gamma\vert d(\zeta,X)\leq \varepsilon\}$. Then by Lemma \ref{translateclose}, we may for all $\varepsilon>0$ write 
$$X(\varepsilon)\subseteq \Gamma^{-(t-1)} \cup\bigcup_{Q\in \calF[p^{r}]\setminus \calF[p^{r-1}]} X(\varepsilon)\cap T_{Q}X(\varepsilon)$$
as well as $X\cap\Gamma^{div}\subseteq \Gamma^{-(t-1)} \cup\bigcup_{Q\in \calF[p^{r}]\setminus \calF[p^{r-1}]} X\cap T_{Q}X\cap\Gamma^{div}$. Furthermore, for our ultrametric distance we get the slightly counter-intuitive consequence of Lemma \ref{lemma:capepsilon} that we may rewrite $X(\varepsilon)\cap T_{Q}X(\varepsilon)=X\cap T_{Q}X(\varepsilon)$. Since $\calF[p^{r}]\setminus \calF[p^{r-1}]$ is a finite set, it therefore suffices to prove the result for each of the formal subschemes $X\cap T_QX$. Since $Q$ is by construction not on the stabilizer, $X\cap T_QX$ is a closed formal subscheme of lower dimension defined over a finite extension of $R$. The result now follows by induction and the stronger statement holds as well by induction provided $\Sigma$ is discrete in the $p$-adic topology on the $n$-dimensional formal disk. 
\end{proof}
We may deduce a stronger result when we are able to take quotients by positive dimensional stabilizers and induct. For the formal torus $\widehat{\mathbb{G}}_m^n$, we know enough about the structure of formal subgroups to carry this through: 
\begin{lemma}\label{lemma:toruschars}
Let $X^*(\widehat{\IG}_m^n):=\Hom_{frm.grp}(\widehat{\IG}_m^n,\widehat{\IG}_m)\cong \IZ_p^n$ denote the (formal) characters on the formal torus. Then any (closed, affine) formal subtorus is the vanishing locus of a $\mathbb{Z}_p$-submodule $M$ of $X^*(\widehat{\IG}_m^n)\cong\IZ_p^n$. 
\end{lemma}
\begin{proof}
This can be deduced from the Manin--Mumford type result for the formal torus, but we give a direct proof. We proceed by induction on $n$, the result holding for $n=1$. It also by induction suffices to show that if a subtorus $G\hookrightarrow \widehat{\IG}_m^n $ is proper, there exists some character $\chi\in X^*(\widehat{\IG}_m^n)$ vanishing on $G$. To prove the claim, we consider the projection morphism $\pi:\widehat{\IG}_m^n\to \widehat{\IG}_m^{n-1}$. If $\pi\vert_G$ is not surjective, we conclude by induction applied to $\pi(G)\hookrightarrow \widehat{\IG}_m^{n-1}$. We therefore assume $\pi\vert_G$ is a surjective morphism of formal groups. If $\ker(\pi)$ was infinite, then by a Zariski-closure argument the closed formal subgroup $G$ would not be proper. Hence $\ker(\pi\vert_G)$ is finite and we may assume it trivial after replacing $G$ by $[p^m]G$ for some $m$. It suffices to prove the lemma for $[p^m]G$ and we may therefore assume $\pi\vert_G$ is an isomorphism. But then we must have $(\pi\vert_G)^{-1}(X_1,\ldots,X_{n-1})=(X_1,\ldots,X_{n-1}, \phi(X_1,\ldots,X_{n-1}))$ for some $\phi\in X^*(\widehat{\IG}_m^{n-1}))$ so that $\phi(X_1,\ldots,X_{n-1})=X_n$ on $G$. 
\end{proof}
One then has the stronger result for formal tori: 
\begin{theorem}\label{thm:backwardtorus}
Let $X\hookrightarrow\widehat{\mathbb{G}}_m^n$ denote a closed affine formal subscheme over a complete, discretely valued subring $R$ of $\calO_{\IC_p}$. Let $\Gamma\subset\widehat{\IG}_m^n(\overline{K})$ be a finitely generated $\IZ_p=\End(\widehat{\IG}_m)$-module and let $\Sigma\subseteq \Gamma^{div}$ be any subset of the divisible hull. There exists a finite union of translates $\calT_i$ of positive dimensional formal subtori by elements in $\Sigma$ and a fixed $m\in \IN$ such that: 
$$X\cap(\Sigma\setminus\Gamma^{-m})=\bigcup_{i=1}^f \calT_i\cap (\Sigma\setminus\Gamma^{-m}).$$
\end{theorem}
\begin{proof}
It suffices to prove that if $X$ is irreducible passing through the origin and that $\Sigma\setminus \Gamma^{-m}$ is dense for some fixed $m$, then $X$ is a formal subtorus. We proceed by induction on $n$, the result holding for $n=1$ by Weierstrass preparation. By Theorem \ref{thm:generalbackward}, we know $X$ contains a formal subtorus $H$ of dimension $k\geq 1$. 
We may as well assume that its annihilator characters $M_H\subset X^*(\widehat{\IG}_m^n)=\IZ_p^n$ as in Lemma \ref{lemma:toruschars} form a free $\mathbb{Z}_p$-module of rank $n-k$. Moreover, after applying a $p$-power isogeny to $X$ we may as well assume the quotient $\IZ_p^n/M_H$ is torsion-free and that we arrive at a split exact sequence
$$0\to\IZ_p^k\to \IZ_p^n\to M_H\to 0.$$
We get corresponding maps of formal tori and in particular denote by
$$\pi_k:\widehat{\IG}_m^n\to \widehat{\IG}_m^{n-k}$$ 
the projection corresponding to the section $M_H\to\IZ_p^n$. By induction, there is some $m$ such that the Zariski closure of $\pi_k(X\cap (\Sigma\setminus \Gamma^{-m}))$ is a finite union of translates of formal subtori $\calT_i$. If none of them is positive dimensional, $X$ is contained in a finite union of hyperplanes and we conclude by induction. If however one of them is positive dimensional, then $X$ contains a formal torus translate $H'$ of strictly greater dimension than $k$. The result follows, since we may replace $H$ by $H'$ and rerun the same argument, increasing $k$ until $\pi_k(X)$ is finite.

\end{proof}
Such a stronger result can of course be obtained for more general formal groups $\calF$ when one is able to run a similar inductive argument.

\subsection{Forward orbits: a counterexample to $p$-adic formal Mordell--Lang}


We have essentially reduced the problem of studying the Zariski-closures of subsets of the divisible hull of $\Gamma$ to just the study of closures of subsets of $\Gamma$. It is natural to ask whether a version of the Mordell-Lang Conjecture holds in this setting. Namely, stating the question for formal tori: 
\begin{question}
Let $\Gamma<\widehat{\mathbb{G}}_m^n(\mathbb{C}_p)$ denote an arbitrary finitely generated subgroup and let $\Gamma^{div}$ denote the divisible hull of $\Gamma$. Is any closed formal subscheme $X\subset \Spf(R[[X_1,\ldots,X_n]])$ with $X\cap \Gamma^{div}$ Zariski-dense a finite union of translates of subtori by elements in $\Gamma^{div}$ ?
\end{question}
Proofs in the classical setting rely heavily on Diophantine approximation results such as the Thue--Siegel--Roth theorem (see e.g., \cite{AST_1975Liardet}) and on the algebraicity of the functions involved. For instance Laurent's proof (\cite{MR767195}) relies heavily on the finite support of the coefficients of polynomials defining the algebraic subvariety. We proceed in this section to answer this question in the \textbf{negative} by exhibiting counter-examples when the underlying functions are not algebraic but merely analytic. This is in contrast to the results for backward iterates under non-invertible endomorphisms of formal groups of Section \ref{subsec:backward}. \par

Recall that $[a](\mathbf{X})$ denotes the endomorphism of the ambient formal group corresponding to $a\in \IZ_p$. We shall focus on a special case where  our $n$-dimensional formal group is a power $\calG^n$ of a one-dimensional formal group of finite height over $R$ such as the formal multiplicative group.
It is then natural to consider $$\Gamma_{\IZ_p}:=\{[k_1](\pi_1),\ldots,[k_n](\pi_n)\vert k_i\in \mathbb{Z}_p\}\subset \calG^n(\overline{K}),$$
acting via endomorphisms $\IZ_p\hookrightarrow \End(\calG)$ in each coordinate and for fixed $\pi_i\in \im_{\overline{K}}$. However, it is too ambitious to ask that an irreducible component of the (formal) Zariski-closure of a subset $\mathcal{S}\subset\Gamma$ be the translate of a (closed, affine) formal subgroup of $\calG^n$. Indeed, as was pointed out to us by Congling Qiu, if the $\pi_i$ are not preperiodic points for the endomorphisms of $\calG$, this set $\Gamma$ is simply too big. For instance for $\widehat{\IG}_m^2/\IZ_p$ and $\pi_1=\pi_2=p$ we have that $(1+p)^{\IZ_p}-1\cong p\IZ_p$ and therefore polynomials such as $\phi(X,Y)=X-Y-p$ already vanish at infinitely many points in $\Gamma$.  \par
It therefore seems more reasonable to restrict to forward orbits via iterations of the formal group law and consider: 
 $$\Gamma:=\{[k_1](\pi_1),\ldots,[k_n](\pi_n)\vert k_i\in \mathbb{Z}\}\subset \calG^n(\overline{K}).$$
 Since points in $\Gamma$ can now be viewed as forward iterates of $\pi_i$ under power series giving multiplication by some integer in the formal group $\calG$, the truthfulness of a Mordell-Lang type result for such $\Gamma$ should be connected to dynamical Mordell--Lang type results (see e.g., \cite{GhiocaBellTuckerdynamicalMLbook}). Indeed we obtain the following result inspired by a counter-example in dynamics \cite[Prop. 7.1.]{BenedettoGhiocaGapPrinciple} and which in particular shows that no $p$-adic Mordell--Lang type result for forward orbits is possible. 
\begin{prop}\label{prop:counterexample}
 Let $\calF^2$ denote two copies of a one-dimensional $p$-divisible formal group over $\IZ_p$ of finite height $h$. Let $\alpha_1,\alpha_2\in p\IZ_p$ and suppose for simplicity $v_p(\alpha_1)=v_p(\alpha_2)$. Let $[n]$ denote the multiplication-by-$n$ endomorphism in the formal group. Then there exists $\phi\in \IZ_p[[X,Y]]$ such that for an infinite set $S$ of pairs $(n,m)\in\mathbb{N}^2$ one has:
$$\phi([n](\alpha_1),[m](\alpha_2))=0$$
and the set $S$ cannot be covered by any finite union of lines, implying that the zeroes of $\phi$ do not come from a formal subgroup translate. 
 \end{prop}
 \begin{proof}
We shall first construct a sequence of functions $\phi_j\in \IZ_p[X]$, together with pairs of indices $(n_j,m_j)$, such that
$$\phi_j([n_i](\alpha_1))=[m_i](\alpha_2)$$
for all $1\leq i\leq j$. We proceed by induction, and assume $\phi_1,\ldots,\phi_j$ constructed as well as $(n_1,m_1),\ldots,(n_j,m_j)$ chosen. A suitable choice for $(n_1,m_1)$ will be specified later. We now set: 
\begin{align*}
    \phi_{j+1}(X)&:=\phi_j(X)+c_j\psi_j(X)\\
    \psi_j(X)&:=\prod_{i=1}^j(X-[n_i](\alpha_1)).
\end{align*}
In order to achieve $\phi_{j+1}([n_{j+1}](\alpha_1))=[m_{j+1}](\alpha_2)$, it therefore suffices to set the correction factor $c_j$ to 
$$c_j=\frac{[m_{j+1}](\alpha_2)-\phi_j([n_{j+1}](\alpha_1))}{\psi_j([n_{j+1}](\alpha_1))}.$$
A priori we have that $c_j\in \IQ_p$ but we proceed to show that the integers $(n_j,m_j)$ can be chosen so that $c_j\in \IZ_p$. Assume now that $(n_j,m_j)$ are chosen so that the following bounds hold: 
\begin{equation}\label{abscond1}
    \vert [n_{j+1}](\alpha_1)\vert_p<\prod_{i=1}^j\vert [n_i](\alpha_1)\vert_p
\end{equation}
\begin{equation}\label{abscond2}
    \vert [m_{j+1}](\alpha_2)-\phi_j(0)\vert_p\leq \prod_{i=1}^j\vert [n_i](\alpha_1)\vert_p.
\end{equation}
We then have that $\vert \psi_j([n_{j+1}](\alpha_1)) \vert_p=\prod_{i=1}^j\vert [n_i](\alpha_1)\vert_p$ by the properties of the absolute value. Moreover, we may bound: 
\begin{align*}
   \vert [m_{j+1}](\alpha_2)-\phi_j([n_{j+1}](\alpha_1))\vert_p&\leq \max\{\vert \phi_j(0)- \phi_j([n_{j+1}](\alpha_1))\vert_p, \vert \phi_j(0)- [m_{j+1}](\alpha_2)\vert_p\}\\
   &\leq \max\{\vert [n_{j+1}](\alpha_1)\vert_p,\vert \phi_j(0)- [m_{j+1}](\alpha_2)\vert_p\}\\
     &\leq \prod_{i=1}^j\vert [n_i](\alpha_1)\vert_p.
\end{align*}
It therefore follows that $c_j\in \IZ_p$ and $\phi_{j+1}\in \IZ_p[X]$. Furthermore, we see that given any $m\geq 0$, the coefficient of $X^m$ in $c_j\psi_j(X)$ approaches zero as $j\to\infty$. Thus the sequence $\phi_j$ converges coefficient-wise and uniformly on the open ball to some $\phi\in \IZ_p[[X]]$. We therefore have the desired property $\forall j\in\IN$:
$$\phi([n_j](\alpha_1))=[m_j](\alpha_2).$$
It remains to check that we may arrange for the inequalities \eqref{abscond1} and \eqref{abscond2} to hold outside of any finite union of lines in $\IN^2$. For instance, for \eqref{abscond1} we may take $n_j=p^{2^j-1}$ for $j\geq 1$. \par 
In order to satisfy \eqref{abscond2}, assume first we have arranged for $\phi_j(0)\in \alpha_2\IZ_p\setminus\alpha_2^2\IZ_p$. Indeed, this can be done for $\phi_1(0)$ by setting $m_1=1$ and $\phi_1(X)=X-[n_1](\alpha_1)+\alpha_2$. It then follows inductively that $\phi_j(0)\in \alpha_2\IZ_p\setminus\alpha_2^2\IZ_p$ since $\vert c_j \psi_j(0)\vert_p$ is small for our choice of $n_j$. Taking the logarithm of the formal group, we see that there exists $A_j\in\IZ_p$ such that the associated endomorphism of $\calF$ satisfies
$$[A_j](\alpha_2)=\phi_j(0).$$
Since integers are dense in $\mathbb{Z}_p$, we may find $m_{j+1}$ such that $\vert m_{j+1}-A_j\vert_p$ is arbitrarily small. Moreover, define $r_j=(m_{j+1}-m_j)/(p^{2^{j+1}-1}-p^{2^j-1})\in\IQ_p$ for $j\in \IN$. We may add a high enough power of $p$ times any unit to such an $m_{j+1}$ and keep $\vert m_{j+1}-A_j\vert_p$ small. Therefore we may as well assume $m_{j+1}$ chosen so that we also inductively have $r_j\neq r_i$ for any $0\leq i\leq j-1$.
Since the exponential of $\calF$ is continuous, for $\vert m_{j+1}-A_j\vert_p$ small enough we then have
$$\vert [m_{j+1}](\alpha_2)-\phi_j(0)\vert_p=\vert [m_{j+1}](\alpha_2)-[A_j](\alpha_2)\vert_p\leq \prod_{i=1}^j\vert [n_i](\alpha_1)\vert_p,$$
 verifying the inequality \eqref{abscond2}. \par 
Finally, it remains to establish that the set $\{(n_j,m_j):j\in\IN\}$ can not be covered by a finite union of lines. Observe that in the infinite sequence $(n_j,m_j)$ constructed the valuations of $n_j$ grow exponentially, whereas those of $m_j$ are bounded given that the valuations of $\phi_j(0)$ are bounded and  \eqref{abscond2} holds. Suppose an infinite sequence of pairs $(n_j,m_j)$ satisfy for fixed $B_1,B_2,C\in \IZ_p$ a linear equation: 
$$B_1\cdot n_j=B_2\cdot m_j+C.$$
It must then be that the ratio $r_j=(m_{j+1}-m_j)/(p^{2^{j+1}-1}-p^{2^j-1})$ viewed as a sequence in $\IQ_p$ is constant on that infinite sequence. But we constructed $m_j$ such that $r_j$ cannot have such behavior.

\end{proof}
\begin{remark}
    It is clear from the above that such counter-examples exist more generally over discretely valued subrings of $\calO_{\IC_p}$ beyond $\IZ_p$ and the restriction was merely for simplicity of exposition.  
\end{remark}

For the dynamical Mordell--Lang problem on Noetherian spaces, one can observe in \cite{GhiocaBellTuckerNoetherian,BenedettoGhiocaGapPrinciple} a certain gap principle: even though subschemes $V$ with a dense set of iterates need not be periodic, it can be shown that if they aren't periodic there is only a sparse set of iterates lying on $V$. In particular the indices of iterates on $V$ have upper density $\mu=0$ in $\IZ$. These results rely on a version of Szemeredi's theorem (see \cite[Lemma 2.1.]{GhiocaBellTuckerNoetherian}).\par
A straightforward generalization of such results to parameters $\mathbf{k}\in \IZ^n$ can be carried through in our setting, but cannot beat the trivial bound (obtained, say, by Weierstrass preparation) of $O(K^{n-\codim V})$ special points on $V$ having exponents $\{\mathbf{k}\in \IZ^n: \vert k_i\vert \leq K\}$ in some box. It would be interesting to see if one can improve on this and obtain a more general sparsity result. This would also improve potential bounds in applications such as the ones outlined in Section \ref{section:applications}.

\section{Density of classical points on $p$-adic families}\label{section:applications}

We now turn to how such unlikely intersection results can be applied to $p$-adic variational techniques. Indeed, given a $p$-adic analytic family interpolating \emph{classical} algebraic objects such as systems of Hecke eigenvalues or Galois representations associated to automorphic forms, it is interesting and useful to understand when these objects constitute a dense subset of the family (for the appropriate $p$-adic topology). In many constructions, the families in question are finite over a $p$-adic formal disk $\Spf (R[[X_1,\ldots,X_n]])$, and from our unlikely intersection results we shall derive some obstructions to the interpolation of a dense set of so-called \emph{classical} points. Due to counter-examples such as Proposition \ref{prop:counterexample}, the applications are in general not as striking as one might have hoped, but we discuss which non-trivial consequences our unlikely intersection results do yield as well as their limitations. 
We outline in particular some applications concerning the nearly ordinary families of cohomological automorphic forms constructed by Hida, for instance for the group $\GL_2$ over arbitrary number fields.  \par
\subsection{$p$-adic weights}
We shall denote by $T$ a split torus in what follows. A \emph {$p$-adic weight }is a continuous group homomorphism
$$\chi:T(\mathbb{Z}_p)\to \overline{\IQ}_p^\times$$
which we may view as a $\overline{\IQ}_p$-point on \textbf{weight space} $\Spf(\Lambda)$ for the completed group ring
$$\Lambda:=\mathbb{Z}_p[[T(\IZ_p)]].$$
More generally, given a compact abelian $p$-adic analytic group $H$, the functor associating to an Artinian local  $\IZ_p$-algebra $A$ the set of continuous group homomorphisms $\Hom_{cts}(H, A^\times)$ is pro-representable by the formal scheme $\Spf(\mathbb{Z}_p[[H]])$. 
\par
Hida theory produces (nearly) ordinary Hecke algebras $\mathbf{h}^{n.ord}$ which are finite modules over these Iwasawa algebras $\Lambda$ and which parametrize $p$-adic families of automorphic forms by weight (we give some examples below).
\begin{remark}\label{rem:translationinvarianceweights}
In fact, given a reductive group $G_{/\IQ}$, since automorphic forms are invariant under left translation by $G(\IQ)$, it is natural to take as weight space the smaller completed group ring $\mathbb{Z}_p[[T(\IZ_p)/Z_p]]$ where $Z_p$ denotes the closure in $T(\mathbb{Z}_p)$ of the image of $Z(\IQ)\cap K^p\cdot T(\IZ_p)$, where $K^p\subset G(\mathbb{A}_f)$ is an appropriate choice of level away from $p$ and $Z$ denotes the center of $G$. This does not affect our results much: since any compact abelian $p$-adic analytic $H$ is isomorphic to $\IZ_p^d\times F$ for a finite group $F$, the Iwasawa algebras in question are finite products of formal power series rings. We note however that their Krull dimension then depends on the veracity of the Leopoldt conjecture.
\end{remark}
These $p$-adic families can interpolate classical automorphic forms at weights of the following form: 
\begin{defn}
A $p$-adic weight (or point $\chi\in \Spf(\Lambda)$) is called \textbf{locally algebraic} 
if there exists a subgroup $H<T(\mathbb{Z}_p)$ of finite index such that the resulting character 
$$\chi\vert_H: H\to\overline{\IQ}_p^\times$$
is \textbf{algebraic}, meaning there exist integers $(k_1,\ldots,k_d)\in\mathbb{Z}^d$ such that one has $\chi\vert_H(t_1,\ldots,t_d)=t_1^{k_1}\cdots t_d^{k_d}$ for topological generators $(t_1,\ldots,t_d)$ of $H$. 
\end{defn}

We then have the following corollary of our rigidity results: 
\begin{corollary}\label{cor:weights}
Let $\calZ$ denote an irreducible component of the Zariski-closure in weight space $\Spf(\Lambda)$ of a set of locally algebraic points. Then either $\calZ$ is the translate of a formal subtorus by a locally algebraic point, or there exists a fixed power $k>0$ such that algebraic points are Zariski-dense in $[p^k](\calZ)$, where $[p^k](\mathbf{\alpha})=(1+\mathbf{\alpha})^{p^k}-1$ denotes the corresponding endomorphism of the formal torus.
Furthermore, there exist such $\calZ$ which are not translates of formal subtori but interpolate a dense set of algebraic weights. 
\end{corollary}
\begin{proof}
Weight space is a finite union of $p$-adic polydisks $\Spec(R[[X_1,\ldots,X_d]])$ and therefore it is enough to consider each of these. The locally algebraic points on $\Spec(R[[X_1,\ldots,X_d]])$ are, fixing a choice of topological generators $(1+\gamma_1),\ldots,(1+\gamma_d)$ of $1+p\mathbb{Z}_p$, given by points 
$$P_\chi=((1+\gamma_1)^{k_1}\zeta_1-1, \ldots, (1+\gamma_d)^{k_d}\zeta_d-1)$$
for $k_i\in \IZ$ and $p$-power roots of unity $\zeta_i\in \mu_{p^\infty}$ for every $1\leq i\leq d$. They therefore lie in the divisible hull of $\Gamma=\{((1+\gamma_1)^{k_1}-1,\ldots,(1+\gamma_n)^{k_n}-1)\vert \mathbf{k}\in \IZ_p^n\}$ as a subgroup of $\widehat{\IG}_m^d$. The result therefore follows from Theorem \ref{thm:backwardtorus} and Proposition \ref{prop:counterexample}.

\end{proof}
\subsection{Density on Hida families for $\GL_2/F$}
We let $F$ denote a number field of signature $(r_1,r_2)$ and $p$ a prime that splits completely in $F$. Consider a cuspidal automorphic form $\pi$ for $\GL_2(\mathbb{A}_F)$ of some tame level $S\subset \GL_2(\calO_F\otimes\widehat{\IZ})$ which contributes to the cohomology of
the smooth (assuming $S$ is small enough) manifold of dimension $2r_1+3r_2$
$$Y(S)=\GL_2(F)\backslash \GL_2(\mathbb{A}_F)/S\cdot(F\otimes \IR)^\times \cdot K_\infty, $$
where $K_\infty$ denotes the standard maximal compact subgroup of the connected component of the identity at infinity. 
Provided $\pi$ is nearly ordinary above $p$, Hida \cite{MR1313784} proves $\pi$ deforms to a nearly ordinary $p$-adic family by constructing a universal Hecke algebra $\mathbf{h}^{n.ord}(S)$ as a finite module over the completed group ring 
$$\Lambda= R[[\mathbb{G}]]\text{ for }\mathbb{G}:=(\calO_F\otimes \IZ_p)^\times\times(\calO_F\otimes \IZ_p)^\times/ \overline{(S\cap F)^\times}$$
 and for $R$ finite over $\IZ_p$. In other words, as in Remark \ref{rem:translationinvarianceweights}, the group parametrizing the relevant $p$-adic weights is $\mathbb{G}=T(\mathbb{Z}_p)/\overline{(S\cap F)^\times}$ for a maximal $F$-split torus $T$. We note that $\dim(\Lambda\otimes \IQ_p)=[F:\IQ]+r_2+1+\delta_F$ for such $S$, where $\delta_F$ denotes the Leopoldt defect for $F$. Letting $I$ denote the set of embeddings $\sigma:F\hookrightarrow\IC$, the algebraic characters of the torus $T$ can be identified as $X^*(T)\cong\IZ[I]\times\IZ[I]$. Thus, by restricting to a small enough neighborhood of the identity, we can uniquely associate to any locally algebraic weight $\chi$ on $T(\mathbb{Z}_p)/\overline{(S\cap F)^\times}$ a pair $(n(\chi), v(\chi))\in \IZ[I]\times\IZ[I]$ (satisfying a relation on $(S\cap F)^\times$, e.g. $n+2v$ is trivial in Hida's normalizaton). We also associate a finite order character $\epsilon_\chi=\chi(g)g^{-(n(\chi,v(\chi))}$ of $\mathbb{G}$. Hida then proves the relevant control theorem: 
\begin{theorem}[\cite{MR1313784}, Theorem 3.2.]\label{thm:Hidacontrol}
Let $\chi\in\Spf(\Lambda)(\overline{\IQ}_p)$ be a locally algebraic weight and $P_\chi\in\Spec(\Lambda)$ the corresponding prime. If $n(\chi)\geq 0$, then the natural specialization map $$\mathbf{h}^{n.ord}(S)/P_\chi\cdot  \mathbf{h}^{n.ord}(S)\to \mathbf{h}^{n.ord}_{\chi}(S\cap U_0(p^\alpha), R)$$
has finite kernel and cokernel, with $\alpha$ the smallest integer such that $\epsilon_\chi$ factors through $\mathbb{G}/p^\alpha\mathbb{G}$ and $U_0(p^\alpha)$ denoting the subgroup $\{g\in\GL_2(\calO_F\otimes\widehat{\IZ})\vert g_p=\begin{pmatrix}a& b\\c& d\end{pmatrix}\text{ with } c\in p^\alpha\cdot (\calO_F\otimes\IZ_p)\}$ .
\end{theorem}
Here $\mathbf{h}^{n.ord}_\chi(S\cap U_0(p^\alpha), R)$ denotes the usual subalgebra of the endomorophisms of the nearly ordinary part of the parabolic cohomology
$$\End_R\left( H^{r_1+r_2}_{P,n.ord}(Y(S\cap U_0(p^\alpha)), \mathcal{L}(\chi,R[1/p]/R))\right)$$
generated by the Hecke operators and the action of $\mathbb{G}$. The coefficients $\mathcal{L}(\chi,A)$ denote the locally constant sheaf associated to the $\GL_2(\calO_F\otimes\IZ_p)$-representation for $n(\chi)\geq 0$ (call such weights \textbf{arithmetic}) given by acting on polynomials in $(X_\sigma,Y_\sigma)_{\sigma\in I}$ homogeneous of degree $n_\sigma$ with $A$-coefficients and twisting by a $v(\chi)$-th power of the determinant. 
In particular, specializing at locally arithmetic weights $\chi$, we recover Hecke eigensystems of classical cohomological automorphic forms of the appropriate level, provided such eigensystems exist. \par
When $F$ is totally real, the Hecke algebras are torsion-free \cite[Theorem II]{Hidatotreal} over $\Lambda$. However, as soon as $F$ has a complex place, it is known that for a cuspidal automorphic form $\pi=\pi_\infty\otimes \pi_f$ which is cohomological, $\pi_\infty$ has to contribute to a unitary representation at the complex places, and therefore we only find characteristic zero cuspidal cohomology when the coefficients (and thus the weights) are conjugate self-dual (see, e.g, \cite[Lemma 1.3.]{BorelCasselman} and \cite[Thm I.5.3.]{MR1721403}). We call such weights $\chi\in\Spec(\Lambda)(\overline{\IQ}_p)$ \textbf{locally parallel}. This leads to Hecke algebras being torsion over $\Lambda$ as soon as $r_2>0$, and in fact Hida conjectures (\cite[Conjecture 4.3.]{MR1313784}): 
\begin{conj}[Hida]
We have $\dim (\mathbf{h}^{n.ord}(S)\otimes \IQ_p)=[F:\IQ]+1$, so that the support of $\mathbf{h}^{n.ord}(S)$ in $\Lambda$ has codimension $r_2$. 
\end{conj}
Let $\Delta\subset \Spec(\Lambda)(\overline{\IQ}_p)$ denote the locally parallel subset of the arithmetic weights. We thus arrive at the interpolation question: 
\begin{question}\label{question:density}
For a component $\calZ$ of $\Supp_{\Lambda}\mathbf{h}^{n.ord}(S)$, when is $\calZ\cap \Delta$ Zariski-dense in $\calZ$ ?
\end{question}
In particular, in cases with a negative answer to question \ref{question:density} the $p$-adic family does not interpolate a Zariski-dense set of \textbf{classical points} corresponding to Hecke eigensystems of classical cuspidal automorphic forms. \par
We remark that for $\GL_1/F$, the analogous question is addressed by Loeffler in \cite{LoefflerDensity}, noting however that for $\GL_2$ one no longer has an explicit description of the support. We now outline some consequences of our unlikely intersection results to this question:

\begin{theorem}\label{thm:hidafamapplication}
Let $\pi$ denote a cuspidal cohomological automorphic form for $\GL_2/F$ which is nearly ordinary above $p$. If the component $\calH$ of $\pi$ in $\mathbf{h}^{n.ord}(S)$ interpolates a Zariski-dense set of classical automorphic forms, then exactly one of the following holds: 
\begin{enumerate}
    \item The corresponding component $\calZ$ of $\Supp_{\Lambda}\calH$ is a \emph{special} subscheme of $\Spec(\Lambda)$: the translate of a formal subtorus of $\widehat{\IG}_m^n$ by an arithmetic, locally parallel point. 
    \item The finite order characters of arithmetic points appearing on the corresponding component $\calZ$ of $\Supp_{\Lambda}\calH$ have uniformly bounded order, and the scheme $[p^n](\calZ)$ contains a Zariski--dense set of algebraic, parallel weights.
\end{enumerate}
\end{theorem}
\begin{proof}This is more or less a reformulation of results from Corollary \ref{cor:weights}. 
\end{proof}
Let us give a slightly more detailed consequence of the rigidity results assuming for simplicity we are working with $\SL_2/F$. In this case we may just take $\Lambda=R[[(\calO_F\times \IZ_p)^\times]]$ so that weight space is a finite union of $r_1+2r_2$-dimensional $p$-adic polydisks. 

\begin{prop}\label{prop:generalfields}
   Fix a completely split prime $p$ and a tame level $N$ and let $\Gamma_i(M)$ for $i=0,1$ denote the usual congruence subgroups of level $M$ of $\SL_2(\calO_F)$. Then exactly one of the following occurs:
    \begin{enumerate}[(i)]
        \item There exists a constant $C$ uniform in all algebraic weights $\mathbf{k}$ such that the number of cuspidal Hecke eigensystems ordinary at $p$ and new at level $\Gamma_0(N)\cap\Gamma_1(p^n)$ is trivial for all  $\mathbf{k}$ and $n>C$: 
        $$\dim_{\IC}(S^{new,ord}_{\mathbf{k}}(\Gamma_0(N)\cap\Gamma_1(p^n),\IC))=0 \text{ for all }n>C.$$
        \item There exists an arithmetic, locally parallel weight $\mathbf{k}$ and a positive integer $C$ such that 
        $$\dim_{\IC}(S^{new,ord}_{\mathbf{k}}(\Gamma_0(N)\cap\Gamma_1(p^n),\IC))\geq 1 \text{ for all }n>C.$$
    \end{enumerate}
 \end{prop}
 \begin{proof}
  The results follow from Theorem \ref{thm:hidafamapplication}: the first case corresponds to the absence of a positive dimensional formal subtorus translate on a component $\calZ$ of the support of the ordinary Hecke algebra of tame level $N$. Here we use the fact that every $p$-ordinary newform lifts to one of a finite number of $p$-adic families of tame level $N$. In the second case, the presence of a positive dimensional formal subtorus on a component of $\Supp_{\Lambda}\calH$ means that specializing along \textbf{every} weight $\chi\in\Spec(\Lambda)(\overline{\IQ}_p)$ in the kernel of at most $n-1$ fixed cocharacters in $X^*(\widehat{\IG}_m^n)\cong \IZ_p^n$ recovers a classical Hecke eigensystem. It is easy to see this kernel in particular contains weights of Nebentype $p^n$ for all large enough $n$. The lower bound thus follows by specialization and by the control theorem \ref{thm:Hidacontrol}. 
 \end{proof}
 
 When $F$ is imaginary quadratic, a special case of the rigidity results above was used by the author in \cite{Serban2018} to prove certain ordinary families of Bianchi modular forms only interpolate finitely many classical automorphic forms. When $F$ is imaginary quadratic, Hida has proved \cite[Thm 6.2.]{MR1313784} that the support of $\mathbf{h}^{n.ord}$ is equidimensional of codimension $1$ in $\Lambda$. Here weight space can be viewed as a finite number of copies of two-dimenisonal formal tori. Moreover, in this case it is easy to exclude the second option in Theorem \ref{thm:hidafamapplication} given that the closure of any infinte set of algebraic, parallel weights is the diagonal in $\Spec(\Lambda)$, which is a formal subtorus. Therefore if the support $\calZ$ contains a dense set of classical weights, either $\calZ$ is the diagonal or $\calZ$ must contain weights corresponding to classical automorphic forms with Nebentypus for every torsion point on a one dimensional formal subtorus as in Proposition \ref{prop:generalfields} part (ii). \par 
 For concrete families, both of these cases can then be excluded by computation of spaces of Bianchi modular forms to prove finiteness of classical points on a given Hida family, when it holds (though in practice there are computational cost obstacles). Finally, we note that it can indeed happen that the support $\calZ$ is the diagonal, for instance by base-changing a Hida family of modular forms to $F$. It would be interesting to establish whether outside of components arising from functoriality constructions such as base-change or automorphic induction, only finitely many classical automorphic forms lie on a $p$-adic family. We refer the reader to \cite{Serban2018} for details, but part of the motivation for this paper was to provide tools to investigate such questions more broadly. 
 

\subsection{More general density questions}
We end with some considerations beyond the density in $p$-adic families of classsical automorphic forms for $\GL_1$ or $\GL_2$. Letting more generally $G=\Res_{F/\IQ}G_0$ denote the Weil restriction of a non-split inner form $G_0$ of $\GL_N/F$, Hida constructs in \cite{HidaGLN} $p$-adic nearly ordinary families of automorphic forms for the reductive group $G/\IQ$. Letting $T/\IZ$ denote a maximal split torus of $\GL_N/\IZ$, Hida constructs universal $p$-adic Hecke algebras $\mathbf{h}^{n.ord}$  which are finite modules over the completed group ring 
$$\Lambda=\IZ_p[[T(\calO_F\otimes \IZ_p)/\overline{\calO_F^\times}]],$$
together with a control theorem \cite[Theorem 6.5.]{HidaGLN} stating one can recover for $\chi \in X^*(T)$ the Hecke algebra of weight $\chi$ up to finite error by specializing $\mathbf{h}^{n.ord}$ along the corresponding prime $P_\chi\in \Spec(\Lambda)$. Hida conjectures \cite[Conjecture 7.1]{HidaGLN} the Hecke algebra has, for small enough level away from $p$, codimension equal to the defect 
$$\ell_0=\rank (G_\infty)-\rank (Z_\infty K_\infty),$$
where $G_\infty:=\GL_N(F\otimes_\IQ\IR)$, $K_\infty\subset G_\infty$ is a maximal compact, and $Z_\infty$ denotes the center of $G_\infty$. When $l_0\neq 0$, it is known that $\mathbf{h}^{n.ord}$ is torsion over $\Lambda$ (\cite[Corollary 6.24]{Khare2014PotentialAA}), but the full conjecture can be viewed as a non-abelian version of the Leopoldt conjecture (see, e.g., \cite[Corollary 7.4]{HidaGLN} and \cite[Theorem 1.2]{Khare2014PotentialAA}). Hida also proves for some inner forms of $\GL_3/\IQ$ and $\GL_4/\IQ$ (see \cite[Corollary 6.3, Conjecture 7.1.]{HidaGLN}) that one indeed has in those cases
$$\dim(\Lambda\otimes \IQ_p)-\dim(\mathbf{h}^{n.ord}\otimes \IQ_p)=1$$
as predicted. One can in this more general setting study whether 
Hecke eigensystems coming from cuspidal automorphic forms deform into $p$-adic families with a Zariski-dense set of \emph{classical points}. Corollary \ref{cor:weights} shows that the support $\Supp_\Lambda(\mathbf{h}^{n.ord})$ either contains a formal torus $\calT$ of the appropriate dimension $d$, so that specializations along \emph{every} locally algebraic weight in $\calT(\overline{\IQ}_p)$ recover a classical Hecke eigensystem, or that the closure of the algebraic characters appearing on the support is of full dimension $d$. When the Zariski-density of algebraic, conjugate self-dual weights can be excluded as in \cite{Serban2018}, only the first option remains. \par
The case of $\GL_3/\IQ$ was studied by Ash, Pollack and Stevens \cite{AshPollackStevens}, who exhibited examples of so-called \emph{arithmetically rigid} eigenpackets \cite[Theorem 9.4.]{AshPollackStevens} which do not deform to $p$-adic families with a Zariski-dense set of arithmetic specializations. These were again found by an algebraic rigidity criterion (see \cite[Theorem 8.4.]{AshPollackStevens}) in conjunction with explicit computation. We note that symmetric square lifts to $\GL_3$ of ordinary newforms $f$ do indeed deform to $p$-adic families, lifting the Hida family for $\GL_2$. These lifts are all essentially self-dual (see \cite[Section 8]{AshPollackStevens} for details) and in fact the authors conjecture that outside of the essentially self-dual case, arithmetic rigidity holds (\cite[Conjecture 0.1.]{AshPollackStevens}). It would be interesting to investigate such matters further. \par
On the Galois side, conjecturally the $p$-adic universal Hecke algebras in question can be identified with suitable Galois deformation rings and similar density questions arise.
For instance, inspired by work of Gouv\^ea and Mazur's construction of the infinite fern, B\"ockle \cite{Boeckledensity}, Chenevier \cite{ChenevierInfiniteFern}  and Allen \cite{AllenPolarized} have shown density results for geometric (in the sense of the Fontaine-Mazur conjecture) or automorphic points in deformation spaces. Note however that these results rely on some form of \emph{oddness} (as in \cite[\S 1]{ClozelHarrisTaylor}) hypothesis. In general such questions again are subtle and density does not always hold (see e.g.,  \cite[Thm. 5.1, Cor. 5.2.]{CalegariFontaineMazurII}) nor is it expected (see e.g., \cite[Conjecture 1.3]{MR2461903}). \par
Concretely, over CM fields $F$, Childers \cite{Childers2021GaloisDS} extends work of Calegari and Mazur \cite[Section 7]{MR2461903} and given a fixed, nearly ordinary above $p$ residual representation $\bar{\rho}:\Gamma_{F,\Sigma}\to G(\IF_q)$ satisfying a list of assumptions (including $\bar{\rho}(\Gamma_{F,\Sigma})\supseteq [G,G](\IF_q)$, see \cite[Theorem 1.4.]{Childers2021GaloisDS} for details), constructs deformations
$$\rho:\Gamma_{F}\to G(W(\IF_q)[[X_1,\ldots X_n]])$$ 
of $\bar{\rho}$ which are almost everywhere unramified and nearly ordinary above $p$. Here $G$ denotes a split connected reductive group over the Witt vectors $W(\IF_q)$ and $\Gamma_F, \Gamma_{F,\Sigma}$ denote the absolute Galois group of $F$ and its maximal unramified quotient outside a finite set of primes $\Sigma$, respectively. Note that $n=\frac{[F:\IQ]}{2}\dim \mathfrak{t}^0$ where $\mathfrak{t}^0$ is the Lie algebra of the maximal torus of the derived group $[G,G]$. Furthermore, Childers shows that the $\overline{\IQ}_p$-points on this formal disk of deformations which are geometric with parallel Hodge-Tate weights have Zariski-closure of positive codimension in the rigid analytic Zariski topology on the generic fiber of $\Spf(W(\IF_q)[[X_1,\ldots X_n]])$. In particular, points coming from automorphic forms should be geometric with parallel Hodge-Tate weights and thus have positive codimension. \par 
It would be interesting to see if the stronger result of positive codimension in the formal Zariski topology on $\Spf(W(k)[[X_1,\ldots X_n]])$ can be shown using rigidity methods as a tool, as for instance this would imply finiteness when $G=\GL_2$ and $F$ is imaginary quadratic. In examining these and similar density considerations, additional arguments will be required. Indeed our results show the usefulness and limitations of rigidity or unlikely intersection results. Nevertheless, we hope they provide one of the ingredients needed to shed more light on questions related to the sparsity of genuine automorphic forms when $\ell_0\neq 0$. 

\section*{Acknowledgements}
We would like to thank Frank Calegari, Dragos Ghioca and Harald Grobner for helpful conversations on some of the contents of this paper. We thank Congling Qiu for highlighting an issue with a previous version of this paper that ultimately led to a substantial revision. Finally we thank the host institutions, both the University of Vienna and the Swiss Polytechnical Federal Institute of Lausanne (EPFL), as well as the number theory research groups there, for providing a good environment where the research for this article could be carried out.

\nocite{}
\bibliography{rigidbiblio.bib}

\end{document}